\documentclass[10pt,reqno]{amsart}

\usepackage{color}

\newtheorem{thm}{Theorem}[section]

\newtheorem{corollary}[thm]{Corollary}
\newtheorem{lemma}[thm]{Lemma}

\newtheorem{assumption}[thm]{Assumption}

\theoremstyle{remark}
\newtheorem{remark}[thm]{Remark}

\def\qed{{\hfill $\Box$ \bigskip}}

\def\Xint#1{\mathchoice
{\XXint\displaystyle\textstyle{#1}}%
{\XXint\textstyle\scriptstyle{#1}}%
{\XXint\scriptstyle\scriptscriptstyle{#1}}%
{\XXint\scriptscriptstyle\scriptscriptstyle{#1}}%
\!\int}
\def\XXint#1#2#3{{\setbox0=\hbox{$#1{#2#3}{\int}$}
\vcenter{\hbox{$#2#3$}}\kern-.5\wd0}}

\def\dashint{\Xint-}

\newcommand\aint{-\hspace{-0.38cm}\int}

\newcommand\cbrk{\text{$]$\kern-.15em$]$}}
\newcommand\opar{\text{\,\raise.2ex\hbox{${\scriptstyle
|}$}\kern-.34em$($}}
\newcommand\cpar{\text{$)$\kern-.34em\raise.2ex\hbox{${\scriptstyle |}$}}\,}

\def\<{\langle}
\def\>{\rangle}

\newcommand\bR{\mathbb{R}}

\newcommand\bM{\mathbb{M}}

\newcommand\fR{\mathbf{R}}

\newcommand\cA{\mathcal{A}}
\newcommand\cB{\mathcal{B}}

\newcommand\cF{\mathcal{F}}
\newcommand\cG{\mathcal{G}}

\newcommand\cI{\mathcal{I}}

\newcommand\cL{\mathcal{L}}

\newcommand\cQ{\mathcal{Q}}

\newcommand{\mysection}[1]{\section{#1}
\setcounter{equation}{0}}

\begin{document}

\title[LITTLEWOOD-PALEY INEQUALITY FOR HIGHER ORDER]
{Parabolic Littlewood-Paley inequality for a class of time-dependent  operators of arbitrary order, and applications to higher order stochastic PDE}

\author{Ildoo Kim}
\address{Department of Mathematics, Korea University, 1 Anam-dong, Sungbuk-gu, Seoul,
136-701, Republic of Korea} \email{waldoo@korea.ac.kr}

\author{Kyeong-Hun Kim}
\address{Department of Mathematics, Korea University, 1 Anam-dong,
Sungbuk-gu, Seoul, 136-701, Republic of Korea}
\email{kyeonghun@korea.ac.kr}
\thanks{The research of the second
author was supported by Basic Science Research Program through the
National Research Foundation of Korea(NRF) funded by the Ministry of
Education, Science and Technology (20110015961)}

\author{Sungbin Lim}
\address{Department of Mathematics, Korea University, 1 Anam-dong, Sungbuk-gu, Seoul,
136-701, Republic of Korea} \email{sungbin@korea.ac.kr}

\subjclass[2010]{42B25, 26D10, 60H15, 35G05, 47G30}

\keywords{Parabolic Littlewood-Paley inequality, Stochastic  partial differential  equations, Time-dependent high order operators, Non-local operators of arbitrary order}

\begin{abstract}
In this paper we prove a parabolic version of the Littlewood-Paley inequality for a class of time-dependent local and non-local operators of arbitrary order, and as an application we show this inequality gives a fundamental estimate for the $L_p$-theory of the stochastic partial differential equations.
\end{abstract}

\maketitle

\mysection{Introduction}

 The classical
Littlewood-Paley inequality says (see \cite{Ste}) that  for  any $p\in
(1,\infty)$ and $f\in L_p(\fR^d)$,
\begin{equation}
              \label{LP}
\int_{\fR^d} \left(\int^{\infty}_0 |\sqrt{-\Delta}\, e^{t\Delta}
f|^2dt\right)^{p/2} dx\leq N(p)\|f\|^p_p,
\end{equation}
where
$e^{t\Delta} f(x):=S_t f=p(t,\cdot)*f(\cdot)=\frac{1}{(4\pi t)^{d/2}}\int_{\fR^d} f(x-y)e^{\frac{-|y|^2}{4t}}dy$.
 In \cite{kr94,Kr01} Krylov  proved the following parabolic version,
 in which $H$ is a Hilbert space : for any $p\in [2,\infty),
-\infty\leq a<b\leq \infty$, $f\in L_p((a,b)\times \fR^d,H)$,
\begin{equation}
                      \label{eqn krylov}
 \|(\int^t_a|(\sqrt{-\Delta}e^{(t-s)\Delta} f)(s,x)|^2_{H}\,ds)^{1/2}\|^p_{L_p((a,b)\times \fR^d)} \leq
 N(p)
\||f|_{H}\|^p_{L_p((a,b)\times\fR^d)}.
\end{equation}
Some related works and the significance of the parabolic Littlewood-Paley inequality in the $L_p$-theory of stochastic PDEs will be discussed later.

If $f=f(x)$ and $H=\fR$ then  by (\ref{eqn krylov}) with $a=0$ and $b=2$,
\begin{align*}
&\int_{\fR^d}[\int^1_0 |\sqrt{-\Delta}e^{s\Delta}f|^2
ds]^{p/2}dx
\\&\leq \int_{\fR^d}\int^2_1[\int^t_0 |\sqrt{-\Delta}e^{(t-s)\Delta}f|^2
ds]^{p/2}dtdx \leq 2N(p)\|f\|^p_{L_p(\bR^d)}.
\end{align*}
This and the scaling $(\sqrt{-\Delta} S_t f(c\cdot))(x)=\sqrt{-\Delta}(cS_{c^2t}f)(cx)$ yield (\ref{LP}).
 Hence   (\ref{eqn krylov})  is  a generalization of (\ref{LP}).
 Note that by putting $K_0(t,x)=\sqrt{-\Delta} p(t,x)$, we get $\sqrt{-\Delta}e^{(t-s)\Delta}f=K_0(t-s,\cdot)*f(s,\cdot)$ and therefore (\ref{eqn krylov}) becomes
\begin{equation}
         \label{eqn 2.21.1}
\|(\int^t_a|K_0(t-s,\cdot)*f(s,\cdot)(x)|^2_H
ds)^{1/2}\|^p_{L_p((a,b)\times\fR^d)}\leq N\||f|_{H}\|^p_{L_p((a,b)\times\fR^d)}.
\end{equation}
In this article we extend (\ref{eqn 2.21.1}) to a class of  time-dependent operators. For a wide class of differential operators $A(t)$ with symbol $\psi(t,\xi)$, one can define the kernel
\begin{equation*}
       \label{ker}
p(t,s,x)=p_A(t,s,x)=\cF^{-1}( \exp (\int_s^t \psi(r,\xi))dr))(x)
\end{equation*}  so that the solution of
$$
u_t=A(t)u +f, \quad u(0)=0
$$
is given by
$$
u=\int^t_0 p(t,s,\cdot)*f(s,\cdot) ds.
$$
  We provide a classification of operators $A(t)$  for which  (\ref{eqn 2.21.1})
holds with formally
\begin{equation*}
                  \label{eqn 2.24.1}
                  K_A(t,s,x)=\sqrt{-A(t)}p(t,s,x).
\end{equation*}
More generally, we provide sufficient conditions on measurable functions $K(t,s,x)$ on $\fR^{d+2}$ so that
\begin{equation}
                      \label{eqn our}
 \|(\int^t_a|K(t,s,\cdot)*f(s,\cdot)(x)|^2_H
ds)^{1/2}\|^p_{L_p((a,b)\times\fR^d)}\leq N\||f|_{H}\|^p_{L_p((a,b)\times\fR^d)}
\end{equation}
holds  for any $f\in C^{\infty}_0(\fR^{d+1},H)$ with constant $N$  independent of $f, a$ and $b$.
The functions $K(t,s,x)$ are  assumed to satisfy the conditions
described in Assumptions \ref{assumption1} and \ref{assumption2}.

For concrete examples we introduce the operators
$A_1(t)$ of $2m$-order ($m=1,2,3,\cdots$) and $A_2(t)$ of order
$\gamma\in (0,\infty)$
\begin{equation}
                                                 \label{eqn 2.2.1}
A_1(t)u:=(-1)^{m-1} \sum_{|\alpha|=|\beta|=m}a^{\alpha
\beta}(t)D^{\alpha+\beta}u, \quad \quad
 A_2(t)u:=-a(t)(-\Delta)^{\gamma/2}
\end{equation}
where $-a(t)(-\Delta)^{\gamma/2}$ is the operator with symbol $-a(t)|\xi|^{\gamma}$ and the coefficients $a(t)$ and $a^{\alpha \beta}(t)$ are
bounded  measurable in $t$  and satisfy the ellipticity conditions
$$
0<\nu< \Re[a(t)]<\nu^{-1},
$$
and
\begin{align*}
\nu |\xi|^{2m} \leq  \sum_{|\alpha|=|\beta|=m}  \xi^\alpha \xi^\beta
\Re\left[a^{\alpha \beta}(t)\right]  \leq \nu^{-1} |\xi|^{2m}.
\end{align*}
Here  $\Re [z]$ is the real part of $z$. Let $p_1(t,s,x)$ and $p_{2}(t,s,x)$ be the kernels related to $A_1(t)$ and $A_{2}(t)$ respectively. We prove that  (\ref{eqn our}) holds with
$$
K_1(t,s,x):=D^m p_1(t,s,x), \quad K_2(t,s,x):=(-\Delta)^{\gamma/4}p_2(t,s,x).
$$
Letting   the function $f$
depend only on $x$, one can obtain elliptic versions of these
results. For instance, we have for any $\gamma \in (0,\infty)$ and
$f\in L_p(\fR^d)$,
\begin{equation*}
\int_{\fR^d} \left(\int^{\infty}_0 |(-\Delta)^{\gamma/2}
e^{-t(-\Delta)^{\gamma}} f|^2dt\right)^{p/2} dx\leq
N(p,\gamma)\|f\|^p_p, \quad \quad \forall\, \gamma\in (0,\infty),
\end{equation*}
which is an extension of    (\ref{LP}), the classical  (elliptic) Littlewood-Paley inequality.

Among  many other examples of (\ref{eqn our}) are  the product $A_1(t)A_2(t)$ and $(-\Delta)^k \mathcal{L}_0(t)$ ($k=0,1,2,\cdots)$, where
\begin{equation}
           \label{c_0}
\mathcal{L}_0(t)u=\int_{\fR^d} \Big(u(x+y)-u(x)-\chi(y)(\nabla u(x),y) \Big) m(t,y) \frac{dy}{|y|^{d+\gamma}},
\end{equation}
 $\gamma \in (0,2)$, $\chi(y)=I_{\gamma>1}+I_{\gamma=1}I_{|y|\leq 1}$, and $m(t,y)\geq 0$ satisfies a certain condition described in Corollary \ref{maintheorem4}. Note that if $m(t,y)\equiv1$ then $\mathcal{L}_0=-(-\Delta)^{\gamma/2}$.

\vspace{2mm}

One of important applications of the parabolic Littlewood-Paley
inequality is  the theory of stochastic partial differential
equations of the type
\begin{equation}
                 \label{eqn 0}
 du=A_i(\omega,t) u \,\,dt +\sum_{k=1}^{\infty}f^k\,\,dw^k_t,
\quad u(0,x)=0.
\end{equation}
 Here
 $f=(f^1,f^2,\cdots)$ is an
$\ell_2$-valued random function depending on $(t,x)$, and $w^k_t$
are independent one-dimensional Wiener processes defined on a
probability space $(\Omega,P)$. The operators $A_i=A_i(\omega,t)$
are defined in (\ref{eqn 2.2.1}), but this
time we allow the coefficients $a(\omega,t)$ and $a^{\alpha \beta}(\omega,t)$ to depend
also on $\omega\in \Omega$. It turns out that if
$f=(f^1,f^2,\cdots)\in L_p(\Omega\times (0,\infty)\times \fR^d, \ell_2)$
satisfies a certain  measurability condition, the solutions of these
problems are given by
\begin{equation}
                                   \label{eqn 222}
u_i(t,x)=\sum_{k=1}^{\infty}\int^t_0 p_i(t,s,\cdot)*f^k(s,\cdot)(x)
dw^k_s,  \quad i=1,2
\end{equation}
where $p_i(t,s,x)$ are introduced above, but
they are random  due to the randomness of the coefficients.
The derivation of formula (\ref{eqn 222}) can be found in
\cite{Kr99} when $A_i=\Delta$, and by repeating the arguments in
\cite{Kr99} one can derive (\ref{eqn 222}) for such $A_i$. By
Burkholder-Davis-Gundy inequality (see \cite{kr95}), we have
\begin{align}
\|D^mu_1 &(t,\cdot)\|^p_{L_p(\Omega\times (0,\infty)\times \fR^d)}\nonumber\\
&\leq
 N(p)\|[\int^t_0|D^mp_1(t,s,\cdot)*f(s,\cdot)(x)|^2_{\ell_2}ds ]^{1/2}
 \|^p_{L_p(\Omega\times (0,\infty)\times \fR^d)}.     \label{eqn 333}
\end{align}
 The corresponding inequality  for
$u_2$ also holds with  $p_2$ and $(-\Delta)^{\gamma/4}$ in place of $p_1$ and $D^m$ respectively.
Actually if $f$ is not random, then $u_1$ and $u_2$  become  Gaussian processes
and  the reverse inequalities also hold. Thus to prove
$$
 D^mu_1, \,\,(-\Delta)^{\gamma/4}u_2\,\,\in
L_p(\Omega\times (0,\infty)\times \fR^d)
$$ and to get a legitimate start
of the $L_p$-theory of stochastic PDEs of type (\ref{eqn 0}), one has to
estimate the right-hand side of (\ref{eqn 333}).
Obviously (\ref{eqn our}) with $K_1$ and (\ref{eqn 333}) imply
\begin{equation}
               \label{expectation wise}
\|D^mu_1(t,\cdot)\|^p_{L_p(\Omega\times (0,\infty)\times \fR^d)}\leq N(p,m)
 \||f|_{\ell_2}\|^p_{L_p(\Omega\times (0,\infty)\times \fR^d)}.
\end{equation}
Using (\ref{expectation wise})  and
following the ideas in \cite{Kr99}, one can
 construct an $L_p$-theory of the general $2m$-order stochastic PDEs. Similarly one can construct an $L_p$-theory of stochastic PDEs with the operator $A_2(\omega,t)$.
  We acknowledge that if the coefficients $a^{\alpha \beta}$ are independent of $t$ then  inequality (\ref{expectation wise}) for high order stochastic PDEs is also introduced in \cite{NVW}  on the basis of $H^\infty$-functional calculus which is far different from our approach.
One of advantages of our approach is that no regularity condition of the coefficients with respect to time variable is required.

\vspace{2mm}

Below is a short description on the related works. As mentioned
above parabolic Littlewood-Paley inequality related to the Laplacian
$\Delta$ was first proved by Krylov in \cite{kr94,Kr01}. This result
is considered as a foundation of the $L_p$-theory of the
second-order stochastic partial differential equations.  Recently the parabolic Littlewood-Paley inequality was proved for the fractional Laplacian $(-\Delta)^{\gamma/2}$, $\gamma\in (0,2)$, in \cite{CL,IK}, and a slight extension of the result of \cite{CL,IK} was   made  to the operator $\mathcal{L}_0(t)$ in \cite{MP} and  to the operator with symbol $-\phi(|\xi|^2)$ in \cite{IPK}, where $\mathcal{L}_0(t)$ is from (\ref{c_0}) and $\phi$ is a Bernstein function satisfying
$$
c^{-1} \lambda^{\delta_1}\phi(t)\leq \phi(\lambda t)\leq c \lambda^{\delta_2}\phi(t), \quad \quad \forall \,\, \lambda, t\geq 1,
$$
$$
\phi(\lambda t)\leq c\lambda^{\delta_3}\phi(t), \quad \quad  \forall \,\,\lambda, t\leq 1
$$
with some constants $c>1$,  $0<\delta_1\leq \delta_2<1$ and $\delta_3\in (0,1]$.  The operators considered in \cite{CL,IK,IPK,MP} are of  order  less than $2$, they (except $\mathcal{L}_0(t)$) do not depend on $t$. The novelty of this article is that it extends existing results which have been proved for lower order operators independent of $t$ to the
time-dependent local and non-local operators of arbitrary order.

\vspace{2mm}

Next we briefly describe our approach to prove (\ref{eqn our}).  We estimate the sharp function of  $(\int^t_a|K(t,s,\cdot)*f(s,\cdot)(x)|^2_H
ds)^{1/2}$ in terms of the maximal function of $|f|_{H}$, then apply Fefferman-Stein theorem and Hardy-Littlewood maximal theorem. The operators considered in \cite{CL,IK,IPK} are the infinitesimal generators of  certain L\'evy processes, and the related  kernels $p(t,x)$ are transition densities of these processes. Thus to estimate the sharp function of $(\int^t_a|K(t,s,\cdot)*f(s,\cdot)(x)|^2_H
ds)^{1/2}$, appropriate bounds of the transition densities can be used as in \cite{CL,IK,IPK}. But for high order operators there is no such related L\'evy process and this method can not be applied. Instead, we modify the idea in \cite{MP} and  make a good use of Parseval's identity which enables us to avoid using estimates of the kernels related to the operators.

\vspace{2mm}

Finally we introduce some notation used in the article. As usual $\fR^{d}$ stands for the Euclidean space of points
$x=(x^{1},...,x^{d})$,  $B_r(x) := \{ y\in \fR^d : |x-y| < r\}$  and
$B_r
 :=B_r(0)$.
 For  multi-indices $\alpha=(\alpha_{1},...,\alpha_{d})$,
$\alpha_{i}\in\{0,1,2,...\}$, $x \in \fR^d$, and  functions $u(x)$ we set
$$
 u_{x^{i}}=\frac{\partial u}{\partial x^{i}}=D_{i}u,\quad \quad
D^{\alpha}u=D_{1}^{\alpha_{1}}\cdot...\cdot D^{\alpha_{d}}_{d}u,
$$
$$
x^\alpha = (x^1)^{\alpha_1} (x^2)^{\alpha_2} \cdots (x^d)^{\alpha_d},\quad \quad
|\alpha|=\alpha_{1}+\cdots+\alpha_{d}.
$$
We also use $D^m_x$ to denote a partial derivative of order $m$  with respect to $x$.
For an open set $U \subset \fR^d$ and a nonnegative integer $n$, we write $u \in C^n(U)$  if $u$ is $n$-times continuously differentiable in $U$. By $C^n_0(U)$ (resp. $C^{\infty}_0(U)$) we denote the set of all functions in $C^n(U)$ (resp. $C^{\infty}(U)$) with compact supports.
The standard $L_p$-space on $U$ with Lebesgue measure is denoted by $L_p(U)$.
Similarly, by $C_0^\infty(\fR^d, H)$ we denote the set of $H$-valued infinitely differentiable functions with compact support.
We use  ``$:=$" to denote a definition.  $a \wedge b = \min \{ a, b\}$,  $a \vee b = \max \{ a, b\}$ and   $\lfloor a \rfloor$ is the biggest integer which is less than or equal to $a$.
By $\cF$ and $\cF^{-1}$ we denote the Fourier transform and the inverse Fourier transform, respectively. That is,
$\cF(f)(\xi) := \int_{\fR^{d}} e^{-i x \cdot \xi} f(x) dx$ and $\cF^{-1}(f)(x) := \frac{1}{(2\pi)^d}\int_{\fR^{d}} e^{ i\xi \cdot x} f(\xi) d\xi$.
For a Borel
set $X\subset \fR^d$, we use $|X|$ to denote its Lebesgue
measure and by $I_X(x)$ we denote  the indicator of $A$.
For a sequence $a= (a_1,a_2,a_3,\ldots)$, we define $|a|_{\ell_2} = \big( \sum_{k=1}^\infty a_k^2 \big)^{1/2}$.
If we write $N=N(a, \ldots, z )$,
this means that the constant $N$ depends only on $a, \ldots , z$.

\mysection{Main results}

In this section we prove (\ref{eqn our}), a generalized version of the  parabolic Littlewood-Paley inequality, under the following conditions on the kernel $K(t,s,x)$ and  provide a classification of operators $A(t)$ for which  (\ref{eqn our}) holds  with $K=K_A$ (see (\ref{eqn 2.24.1})). Three interesting examples related to  the operators $A_1(t), A_2(t)$ and $(-\Delta)^k\cL_0(t)$ are also presented.

\vspace{3mm}

Assumption \ref{assumption1} below is needed to prove (\ref{eqn our}) for $p=2$.

\begin{assumption}              \label{assumption1}
The kernel $K(t,s,x)$ is a measurable function defined on $\fR^{d+2}$ satisfying
\begin{align}
\label{as1}
\int_{s}^\infty |\cF\{K (t,s, \cdot)\}(\xi)|^2  dt \leq C_{0}
\end{align}
with constant $C_0$  independent of $(s,\xi)$.
\end{assumption}

$\\$

Take a constant $c_{2} > \frac{1}{2}$  and denote

\begin{align}            \label{c1c2relation}
c_{3} := \frac{2(d+1)(c_{2}+1)+3}{2(d+2)}.
\end{align}

\begin{assumption}              \label{assumption2}

(i) For almost all $t$ and each $s<t$, $K(t,s,\cdot)$  $D_x K(t,s,\cdot)$  and $\frac{\partial}{\partial t}K (t,s,x)$ are locally integrable functions of $x$.

 (ii) There exist  functions $F_i(t,s,x)$ and  positive constants $\sigma_i,\kappa_i$ ($i=1,2,3)$ and $C$ such that for almost all $t$ and each $s<t$ and $x\in \fR^d\setminus \{0\}$,

\begin{align}        \label{as2}
\big|D_xK(t,s,x) \big| \leq C  \big| (t-s)^{-\sigma_1} F_1\big(t,s,(t-s)^{-\kappa_1} x\big) \big|,
\end{align}

\begin{align}            \label{as3-1}
\big|D^2_xK(t,s,x) \big| \leq C \Big((t-s)^{-\sigma_2} \big|F_2(t,s,(t-s)^{{-\kappa_2}} x) \big| \wedge (t-s)^{{-c_{2}}} \Big),
\end{align}

\begin{align}            \label{as3-2}
\big|\frac{\partial}{ \partial t }D_xK(t,s,x) \big| \leq C \Big((t-s)^{{-\sigma_3}} \big| F_3(t,s,(t-s)^{{-\kappa_3}} x) \big| \wedge (t-s)^{{-c_{3}}}\Big).
\end{align}

\vspace{2mm}

(iii) For these $F_{i}$ ($i=1,2,3$),  we have

\begin{align}        \label{finite1}
\sup_{s<t}\int_{\fR^d} |x|^{\mu_1}|F_1(t,s,x)|^2 dx  < \infty,
\end{align}

\begin{align}           \label{finite2}
\sup_{s<t}\int_{|x| \geq (t-s)^{c_{2}-c_{3}+1 - \kappa_2}} |x|^{\mu_2} |F_2(t,s,x)|^2 dx  < \infty,
\end{align}

\begin{align}               \label{finite3}
\sup_{s<t}\int_{|x| \geq (t-s)^{c_{2}-c_{3}+1 - \kappa_3}} |x|^{\mu_3} |F_3(t,s,x)|^2 dx  < \infty,
\end{align}


%
where  $\mu_i>d+2$ ($i=1,2,3$)  satisfy the following system

\begin{align}\label{algebraic condition}
\left\{
  \begin{array}{c}
    (c_{2}-c_{3}+1-\kappa_{1})\mu_{1}=d(\kappa_{1}+c_{2}-c_{3}+1)+2(c_{2}-c_{3}-\sigma_{1})+3 \\ \\
    (c_{2}-c_{3}+1-\kappa_{2})\mu_{2}=d(\kappa_{2}-c_{2}+c_{3}-1)+2(c_{2}-\sigma_{2}) \\ \\
    (c_{2}-c_{3}+1-\kappa_{3})\mu_{3}=d(\kappa_{3}-c_{2}+c_{3}-1)+2(c_{3}-\sigma_{3}) \\
  \end{array}
\right\}.
\end{align}
\end{assumption}

\vspace{3mm}

\begin{remark}
(i) Suppose
$$
\sup_{s<t}\int_{\fR^d} |F_i(t,s,x)|^2 dx <\infty, \quad  (i=1,2,3)
$$
and
$$
\sup_{s<t}\int_{\fR^d} |x|^{\hat{\mu}_i}|F_i(t,s,x)|^2 dx  < \infty  \quad (i=1,2,3)
$$
with some $(\hat{\mu}_1, \hat{\mu}_2,\hat{\mu}_3)\in \fR^3$. Then  obviously (\ref{finite1})-(\ref{finite3})  hold
  for any $\mu_i\leq \hat{\mu}_i$ ($i=1,2,3$).

\vspace{2mm}
(ii) Suppose, for example, $c_2-c_3+1-\kappa_2=0$. Then in (\ref{algebraic condition}) we are assuming
$$
d(\kappa_{2}-c_{2}+c_{3}-1)+2(c_{2}-\sigma_{2})=0.
$$
In this case, we have a freedom of choosing $\mu_2$, that is we can  choose arbitrary $\mu_2>d+2$ satisfying (\ref{finite1}).

\vspace{2mm}
(iii)
  Put
\begin{align}            \label{delta01}
\delta_{0} := c_{2}-c_{3} +1,\quad \Theta(\theta,\vartheta):=\theta d-2\vartheta.
\end{align}
One can easily check
$$
\Theta(\theta_{1}+\theta_{2},\vartheta_{1}+\vartheta_{2})=\Theta(\theta_{1},\vartheta_{1})+\Theta(\theta_{2},\vartheta_{2}),
$$
and  \eqref{algebraic condition}  is equivalent to
\begin{align}       \label{matrix}
\left[
  \begin{array}{ccc}
    \delta_{0}-\kappa_{1} &  &  \\
     & \delta_{0}-\kappa_{2} &  \\
     &  & \delta_{0}-\kappa_{3} \\
  \end{array}
\right]
\left[
  \begin{array}{c}
    \mu_{1}  \\
    \mu_{2}  \\
    \mu_{3}  \\
  \end{array}
\right]
=
\left[
  \begin{array}{c}
    \Theta(\kappa_{1}+\delta_{0},\sigma_{1}-\delta_{0})+1  \\
    \Theta(\kappa_{2}-\delta_{0},\sigma_{2}-c_{2})  \\
    \Theta(\kappa_{3}-\delta_{0},\sigma_{3}-c_{3})  \\
  \end{array}
\right].
\end{align}
\end{remark}

\vspace{5mm}

Note that to prove (\ref{eqn our}) we may assume $a=-\infty$ and $b=\infty$. Recall $H$ denote a Hilbert space.
Here are the main results of this article. The proofs of Theorems \ref{maintheorem} and \ref{pseudothm} are given in Sections \ref{pfmainthm} and
\ref{pseudothmpf} respectively.

\begin{thm}             \label{maintheorem}
Let $p \geq 2$. Suppose that Assumptions \ref{assumption1} and \ref{assumption2} hold. Then
for any $f\in C^{\infty}_0(\fR^{d+1},H)$,
\begin{align*}
\left\|\Big(\int_{-\infty}^t \big|K(t,s, \cdot) \ast f(s,\cdot)(x)\big|_H^2 ds \Big)^{1/2} \right\|_{L_p(\fR^{d+1})} \leq N \||f|_H\|_{L_p(\fR^{d+1})},
\end{align*}
where $N$ is independent of $f$.
\end{thm}

Let $A(t)$ be a non-positive operator with the symbol $\psi(t,\xi)$, that is
$$
\cF (A(t)u)(\xi)=\psi(t,\xi) \cF(u)(\xi), \quad \forall \, u\in C^{\infty}_0 (\fR^d).
$$
Define the kernel $p(t,s,x)$ by the formula
\begin{align}
p(t,s,x)=I_{0 \leq s<t } \cF^{-1} \Big( \exp \big(\int_s^t \psi(r,\xi))dr \big) \Big)(x).
\end{align}

\begin{thm}         \label{pseudothm}
Fix $p \geq 2$ and $\gamma >0$. Assume there exist  constants $\nu >0$
such that for any multi-index $|\alpha| \leq \lfloor\frac{d}{2}\rfloor+2$,
\begin{equation}
     \label{A1}
      \Re [\psi(t,\xi)]  \leq - \nu |\xi|^\gamma,
\end{equation}
\begin{equation}
                         \label{A2}
 |D^\alpha \psi(t,\xi)| \leq \nu^{-1} |\xi|^{\gamma -|\alpha|}
 \end{equation}
hold for almost every $t>0$ and $\xi\neq 0$. Then for any $f \in C^\infty_0(\fR^{d+1}, H)$
\begin{align*}
\left\|\Big(\int_0^t |\Delta^{\gamma/4} p (t,s,\cdot) \ast f(s,\cdot)(x)|_{H}^2~ds \Big)^{1/2} \right\|_{L_p(\fR^{d+1})} \leq N \| |f|_{H}\|_{L_p(\fR^{d+1})},
\end{align*}
where $N$ depends only on $p,\nu, \gamma$ and $d$.
\end{thm}

For applications of Theorem \ref{pseudothm} we recall the operators $A_i(t)$ from (\ref{eqn 2.2.1}), that is,
$$
A_1(t)u=(-1)^{m-1}\sum_{|\alpha|=|\beta|=m}a^{\alpha \beta}(t)D^{\alpha+\beta}u,\quad \quad A_2(t)=-a(t)(-\Delta)^{\gamma/2}
$$
where the coefficients $a^{\alpha \beta}$ and $a(t)$ are bounded complex-valued measurable functions satisfying $\nu<\Re[a(t)]<\nu^{-1}$ and
\begin{align*}
\nu |\xi|^{2m} \leq  \sum_{|\alpha|=|\beta|=m}  \xi^\alpha \xi^\beta \Re\left[a^{\alpha \beta}(\omega,t)\right]  \leq \nu^{-1} |\xi|^{2m}, \quad \forall \xi\in \fR^d.
\end{align*}
 Denote
$$
p_1(t,s,z)=p_{1,m}(t,s,x)= I_{0 \leq s<t }\cF^{-1} \Big( \exp\big\{-\int_s^t a^{\alpha \beta}(r) \xi^\alpha \xi^\beta dr \big\} \Big)(x),
$$
$$
p_2(t,s,x)=p_{2,\gamma}(t,s,z)=I_{0 \leq s<t }\cF^{-1} \Big( \exp\big\{- |\xi|^{\gamma} \int_s^t a(r)dr\big\} \Big)(x).
$$

\begin{corollary}         \label{maintheorem3}
Let $p\geq 2$. Then for any $f \in C^\infty_0(\fR^{d+1}, H)$,
\begin{align*}
\left\|\Big(\int_0^t |\Delta^{m/2}p_1(t,s,\cdot) \ast f(s,\cdot)(x)|_{H}^2~ds \Big)^{1/2} \right\|_{L_p(\fR^{d+1})} \leq N \| |f|_{H}\|_{L_p(\fR^{d+1})},
\end{align*}
where $N$ depends only on $p,\nu, m$ and $d$.
\end{corollary}

\begin{proof}
It is obvious that the symbol
$\psi(t,\xi)=-a^{\alpha \beta}(t) \xi^\alpha \xi^\beta$
satisfies  (\ref{A1}) and (\ref{A2}) with $\gamma=2m$ and any multi-index $\alpha$. Thus the corollary follows from Theorem \ref{pseudothm}.
\end{proof}

\begin{corollary}         \label{maintheorem2}
Let $p\geq 2$.  Then for any $f \in C^\infty_0(\fR^{d+1}, H)$,
\begin{align*}
\left\|\Big(\int_0^t |\Delta^{\gamma/4} p_{2} (t,s,\cdot) \ast f(s,\cdot)(x)|_{H}^2~ds \Big)^{1/2} \right\|_{L_p(\fR^{d+1})} \leq N \| |f|_{H}\|_{L_p(\fR^{d+1})},
\end{align*}
where $N$ depends only on $p,\nu, \gamma$ and $d$.
\end{corollary}

\begin{proof}
The symbol related to the operator $A_2(t)$ is  $-a(t)|\xi|^{\gamma}$, and therefore the corollary follows from  Theorem \ref{pseudothm}.
\end{proof}
Recall we defined $(-\Delta)^{\gamma/2}$ as the operator with symbol $|\xi|^{\gamma}$ for any $\gamma\in (0,\infty)$.
For further applications of Theorem \ref{pseudothm}, we consider a product of $(-\Delta)^k$ and an integro-differential operator $\cL_0=\cL_{0,\gamma}$. We remark that in place of $(-\Delta)^k$ one can consider many other pseudo-differential or high order differential operators.

\vspace{2mm}

Fix $\gamma\in (0,2)$, and for $k=0,1,2,\cdots$ denote
\begin{align*}
 &\cL_{k}(t)u =(-\Delta)^k\cL_{0,\gamma} u:=\\
&\int_{\fR^d \setminus \{0\}} \Big( (-\Delta)^k u(t,x+y)- (-\Delta)^k u(t,x)-\chi(y)(\nabla (-\Delta)^k u (t,x),y) \Big) m(t,y)
\frac{dy}{|y|^{d+\gamma}}
\end{align*}
where $\chi(y)= I_{\gamma >1} + I_{|y|\leq1}I_{\gamma=1}$  and $m(t,y)\geq 0$ is measurable function satisfying  the following conditions :

(i) If $\gamma=1$ then
\begin{align}           \label{cancel}
\int_{\partial B_1} w m(t,w)~S_1(dw)=0, \quad \forall t >0,
\end{align}
where  $\partial B_{1}$ is the unit sphere in $\fR^d$ and $S_{1}(dw)$ is the surface measure on it.

(ii) The function $m=m(t,y) $ is zero-order homogeneous and differentiable in $y$ up to $d_0 = \lfloor\frac{d}{2}\rfloor+2$.

(iii) There is a constant $K$ such that for each  $t \in \fR$
$$
\sup_{|\alpha| \leq d_0, |y|=1} |D^\alpha_y m^{(\alpha)} (t,y) | \leq K.
$$
It turns out that the operator $\cL_k$ is a pseudo differential operator with  symbol
\begin{align*}
\psi(t,\xi) = - c_1 |\xi|^{2k}\int_{\partial B_{1}} |(w,\xi)|^\gamma [1-i\varphi^{(\gamma)}(w,\xi)] m(t,w)~S_{1}(dw),
\end{align*}
\begin{align*}
\varphi^{(\gamma)}(w,\xi)=c_2\frac{(w,\xi)}{|(w,\xi)|} I_{\gamma \neq 1}- \frac{2}{\pi} \frac{(w,\xi)}{|(w,\xi)|} \ln |(w,\xi)| I_{\gamma=1},
\end{align*}
and $c_1(\gamma,d)$, $c_2(\gamma,d)$ are certain positive constants.

(iv) There is a constant $N_{0}>0$ such that the symbol $\psi(t,\xi)$ of $\cL_k$ satisfies
\begin{equation}
             \label{eqn psi7}
\sup_{t,|\xi|=1} \Re[\psi(t,\xi)] \leq  -N_{0}.
\end{equation}
One can check that  (\ref{eqn psi7}) holds if there exists a constant $c>0$ so that $m(t,y)>c$ on a set  $E\subset \partial B_1$ of positive  $S_1(dw)$-measure.

\begin{corollary}         \label{maintheorem4}
Let $p\geq 2$ and $p(t,s,x)$ be the kernel related to $\cL_k(t)$. Then under above conditions (i)-(iv) on $m(t,y)$ it holds that  for any $f \in C^\infty_0(\fR^{d+1}, H)$
\begin{align*}
\left\|\Big(\int_0^t |\Delta^{k/2+\gamma/4} p(t,s,\cdot) \ast f(s,\cdot)(x)|_{H}^2~ds \Big)^{1/2} \right\|_{L_p(\fR^{d+1})} \leq N \| |f|_{H}\|_{L_p(\fR^{d+1})},
\end{align*}
where $N$ depends only on $p, \gamma, k,d, N_0$ and $K$.
\end{corollary}

\begin{proof}

 Note that for $\xi \neq 0$
 \begin{align*}
\psi(t, \xi)&
=|\xi|^{2k+\gamma}\psi\Big(t, \frac{\xi}{|\xi|} \Big)=: |\xi|^{2k+\gamma} \tilde \psi(t,\xi).
 \end{align*}
The above equality  is obvious if $\gamma \neq 1$, and if  $\gamma =1$ then by \eqref{cancel}
\begin{align*}
\psi(t, \xi)
&=|\xi|^{2k+1} \psi\Big(t, \frac{\xi}{|\xi|} \Big) + |\xi|^{2k}\ln |\xi| \int_{\partial B_1} (w,\xi) m(t,w)~S_1(dw) \\
&=|\xi|^{2k+1} \psi\Big(t, \frac{\xi}{|\xi|} \Big).
\end{align*}
By using condition (iii) one can check (see e.g. \cite[Remark 2.6]{MP92}) that for
 any multi-index $\alpha$, $|\alpha|\leq d_0$,  there exists a  constant $N=N(\alpha)$ such that
$$
|D^\alpha \tilde \psi(t,\xi)| \leq N|\xi|^{-|\alpha|}.
$$
Thus it is obvious that the given symbol $\psi$ satisfies (\ref{A1}) and (\ref{A2}). The corollary is proved.

\end{proof}

\mysection{Some preliminary estimates}

For $f\in C^{\infty}_0(\fR^{d+1},H)$, we define
\begin{align*}
\cG f (t,x):=  \left(\int_{-\infty}^t \big|K(t,s, \cdot) \ast f(s,\cdot)(x)\big|_H^2 ds \right)^{1/2}.
\end{align*}

\begin{lemma}           \label{lemma1}
Let Assumption \ref{assumption1} hold and $f\in C^{\infty}_0(\fR^{d+1},H)$. Then for
 any  $-\infty \leq a \leq b \leq \infty$,
\begin{align}
\label{4023} \|\cG f \|^2_{L_2( (a,b) \times \fR^d)} \leq
N\||f|_H\|^2_{L_2((-\infty,b) \times \fR^d)},
\end{align}
where $N=N(d,C_0)$.
\end{lemma}

\begin{proof}
By the continuity of $f$, the range of $f$ belongs to a
separable subspace of $H$. Thus by using  a countable orthonormal
basis of this subspace and the Fourier transform one easily finds

\begin{align*}
& \|\cG f\|^2_{L_2((a,b) \times \fR^{d})}\\
&= (2\pi)^d \int_{\fR^d} \int_{a}^b  \int_{-\infty}^t |\cF\{K (t,s, \cdot)\}(\xi)|^2 \, |\cF(f)(s,\xi)|^2_H ds dt d\xi\\
&\leq (2\pi)^d \int_{\fR^d} \int_{-\infty}^b \int_{a}^b I_{0 \leq t-s} |\cF\{K (t,s, \cdot)\}(\xi)|^2  dt|\cF(f)(s,\xi)|^2_H  ds d\xi\\
&\leq (2\pi)^d \int_{\fR^d} \int_{-\infty}^b \left(\int_{s}^\infty  |\cF\{K (t,s, \cdot)\}(\xi)|^2  dt\right)|\cF(f)(s,\xi)|^2_H  ds d\xi.
\end{align*}
From \eqref{as1}, we have
\begin{align*}
\|\cG f\|^2_{L_2( (a,b) \times \fR^{d} )} \leq N \int_{-\infty}^b
\int_{\fR^d} |\cF(f)(s,\xi)|^2_H ~d\xi ds.
\end{align*}
The last expression is equal to the right-hand side of (\ref{4023}),
and therefore  the lemma is proved.
\end{proof}

\begin{corollary}               \label{cor1}
Let $r_1,r_2>0$. Suppose that Assumption \ref{assumption1} holds, $f\in C^{\infty}_0(\fR^{d+1},H)$, and $f(t,x)=0$ for $x\not\in  B_{3r_1}$.
Then
$$
\int_{-2r_2}^0 \int_{B_{r_1}}|\cG f(s,y)|^2dyds \leq N(d,C_0) \int_{-\infty}^0 \int_{B_{3r_1}}|f(s,y)|_H^2dyds.
$$
\end{corollary}
\begin{proof}
Applying Lemma \ref{lemma1} with $a = -2r_2$ and $b=0$ and using the condition on $f$, we gt
\begin{align*}
\int_{-2r_2}^0 \int_{B_{r_1}}|\cG f(s,y)|^2dyds
&\leq \int_{-\infty}^0 \int_{\fR^d}|\cG f(s,y)|^2dyds\\
&\leq N\int_{-\infty}^0 \int_{\fR^d}|f(s,y)|_H^2dyds\\
&= N\int_{-\infty}^0 \int_{B_{3r_1}}|f(s,y)|_H^2dyds.
\end{align*}
Hence the corollary is proved.
\end{proof}

For $R \geq 0$ and real-valued locally integrable functions $h(x)$ on $\fR^d$, define the
maximal functions
$$
\bM^R_x h(x) := \sup_{r>R} \frac{1}{|B_{r}(x)|} \int_{B_{r}(x)} |h(y)| dy, \quad \quad \bM_x h(x):=\bM^0_x h(x).
$$
 Similarly,
for  real-valued locally integrable functions $h=h(t)$ on $\fR$ we introduce
$$
\bM^R_t h(t) := \sup_{r>R} \frac{1}{2r} \int_{-r}^r |h(t+s)|\, ds, \quad \quad \bM_t h(t):=\bM^0_t h(t).
$$
For  functions $h=h(t,x)$,  set
$$
\bM^R_x h(t,x) := \bM^R_x(h(t,\cdot))(x), \quad \bM^R_th(t,x) = \bM^R_t(h(\cdot,x))(t).
$$

Obviously if $R_1 \geq R_2$, then
$$
\bM^{R_1}_x h(x) \leq \bM^{R_2}_x h(x)
$$
and if $R_n \downarrow R$, then
$$
\bM^{R_n}_x h(x) \uparrow \bM^{R}_x h(x).
$$
The same properties hold for $\bM^R_t$.

 Let  $S_1(dw)$ denote the counting measure on $\{-1,1\}$ if $d=1$ and the surface measure on the unit sphere if $d \geq 2$. The following lemma is a slight modification of  \cite[Lemma 8]{MP}.

\begin{lemma}           \label{lemma2}
Let $f\in C_0(\fR^d)$, and $v(x)$ be a  locally integrable and continuously differentiable function on $\fR^d$. Let $x,y\in\fR^{d}$,   $|x-y| \leq R_1$ and $f(y-z)=0$ if $|z| \leq R_2$ with some constants $R_1, R_2 \geq 0$
Then it holds that
\begin{align*}
\big|(f \ast v)(y)|
\leq N \big( \bM^{R_1+R_2}_x f^2(x)\big)^{1/2} \int_{R_2}^{\infty}(R_1+\rho)^d \Big( \int_{\partial B_1 }\big( \nabla v(\rho w),w\big)^2S_1(dw)\Big)^{1/2}d\rho,
\end{align*}
where $N=N(d)$.

\end{lemma}
\begin{proof}
 Since the case $d=1$ is easier, we assume $d\geq 2$.  Using the
polar coordinates and Fubini's theorem we get
\begin{align*}
\int_{|z| >R_2 } f(y-z) v(z)~dz
&= \int_{R_2}^\infty \int_{\partial B_{1}}f(y- \rho w) v(\rho w)\rho^{d-1} ~S_1(dw)d\rho \\
&=  \int_{\partial B_{1}} \left[\int_{R_2}^\infty v(\rho w) \left(\frac{d}{d\rho}\int_{R_{2}}^\rho f(y- \gamma w)  \gamma^{d-1}~d\gamma\right) S_1(dw) \right] d\rho.
\end{align*}
By integration by parts and the assumption on $v$, for almost all $w$,
\begin{align*}
&\int_{R_2}^\infty v(\rho w) \left(\frac{d}{d\rho}\int_{R_2}^\rho f(y- \gamma w)  \gamma^{d-1}~d\gamma \right) d\rho     \\
&=-\int_{R_2}^\infty (\nabla v(\rho w), w)\int_{R_2}^\rho f(y- \gamma w)  \gamma^{d-1}~d\gamma d\rho.
\end{align*}
In the above we use the fact that there exists a sequence $\rho_n \to \infty$, which might be dependent on $w$, so that $v(\rho_n w)\to 0$ as $n\to \infty$ and that $\int^{\rho}_{R_2} f(y-\gamma w)\gamma^{d-1} d\gamma $ is a bounded function of $\rho$. Also note that the limits of two improper integrals exist since the first one is actually an integral over finite interval.

By   the assumption $|x-y| \leq R_1$,  for  any $\rho > R_2$
\begin{align*}
\int_{B_\rho} f^2(y-z)~dz =\int_{B_\rho(y)} f^2(z)~dz
&\leq \int_{B_{R_1+\rho}(x)} f^2(z)~dz \\
&\leq  N(d)(R_1+\rho)^d \bM^{R_1+R_2}_x f^2(x).
\end{align*}
Finally using Fubini's theorem, H\"older's inequality, and the assumption that $f(y-z)=0$ if $|z| \leq R_2$, we get
\begin{align*}
 &|(f \ast v) (y) |
 \\&\leq \Big|\int_{R_2}^\infty \int_{\partial B_1}(\nabla v(\rho w), w)\int_{R_2}^\rho f(y- \gamma w)  \gamma^{d-1}~d\gamma S_1(dw)d\rho \Big|\\
&\leq \int_{R_2}^\infty \Big(\int_{\partial B_1} \int_{R_2}^\rho \Big|(\nabla v(\rho w), w)\Big|^2 \gamma^{d-1}d\gamma S_1(dw) \Big)^{1/2}\\
&\hspace{40mm}\times\Big(\int_{\partial B_1}\int_{R_2}^\rho f^2(y- \gamma w)  \gamma^{d-1}~d\gamma S_1(dw)\Big)^{1/2}d\rho \\
&\leq \int_{R_2}^\infty \rho^{d/2}\Big(\int_{\partial B_1} \Big|(\nabla v(\rho w), w)\Big|^2 S_1(dw) \Big)^{1/2} \Big(\int_{ |z| \leq \rho } f^2(y- z)  ~dz\Big)^{1/2}d\rho \\
&\leq N \big( \bM^{R_1+R_2}_xf^2(x) \big)^{1/2}\int_{R_2}^\infty (R_1+\rho)^d\Big(\int_{\partial B_1} \Big|(\nabla v(\rho w), w)\big|^2 S_1(dw) \Big)^{1/2}d\rho.
\end{align*}
The lemma is proved.
\end{proof}

For $r_1,r_2>0$ denote
$$
Q_{r_2,r_1} := (-2r_2,0) \times B_{r_1}.
$$
 \begin{lemma}          \label{lemma3}
Suppose there exist constants $\sigma,\kappa>0$ and $\mu>d+2$ so that
\begin{equation}
                                                         \label{as2'}
\big|D_xK(t,s,x) \big| \leq C  \big| (t-s)^{-\sigma}
F_1\big(t,s,(t-s)^{-\kappa} x\big) \big|,
\end{equation}
\begin{equation}                                       \label{theta1range2}
-2{\sigma} + {\kappa}(\mu +d ) >-1,
\end{equation}
and
\begin{equation}
                                                         \label{finite1'}
H_{1,K}(\mu):=\sup_{s<t}\int_{|x| > r_1 (s-r)^{-\kappa}
}|x|^{\mu}|F_{1}(t,s,x)|^{2}dx < \infty.
\end{equation}
Let $f\in C^{\infty}_0(\fR^{d+1},H)$ with support in $(-10 r_2, 10
r_2) \times \fR^d \setminus B_{2r_1}$. Then for any $x \in B_{r_1}$
we have
$$
\int_{Q_{r_2,r_1}} |\cG f(s,y)|^2 ~dsdy \leq  N H_{1,K}(\mu) r_1^{\mu}   r_2^{-2{\sigma}+{\kappa}(\mu+d)+1} \int_{ - 10r_2}^0  \bM^{3r_1}_x |f|_H^2(s,x)ds,
$$
 where $N=N(d,\mu,{\sigma},{\kappa},C_0,C)$.

\end{lemma}

\begin{proof} Let $x \in B_{r_1}$, $(s,y) \in Q_{r_2,r_1}$ and $r\leq s$. Then
$|x-y| \leq 2r_1$, and $|z| \leq r_1$ implies $|y-z| \leq  2r_1$ and
$f(r,y-z)=0$ due to the assumption on $f$. Therefore,
\begin{align*}
|K(s,r, \cdot) \ast f(r, \cdot)(y)|_H \leq \int_{|z|\geq r_{1}}|K(s,r,z)||f|_H(r,y-z) ~dz.
\end{align*}
 Applying Lemma \ref{lemma2} with $R_1=2r_1$ and $R_2=r_1$, we get
\begin{align}
&|K(s,r,\cdot) \ast f(r, \cdot)(y)|_H^2  \nonumber\\
&\leq  N\bM^{3r_1}_x |f|_H^2(r,x) \left(\int_{r_1}^\infty (2r_1 + \rho)^d
 \big[\int_{\partial B_{1}} \Big| \nabla K (s,r,\rho w) \Big|^2S_{1}(dw) \big]^{1/2} d\rho \right)^2 \nonumber\\
&\leq  N\bM^{3r_1}_x |f|_H^2(r,x) \left(\int_{r_1}^\infty  \rho^d
\big[\int_{\partial B_{1}} \Big| \nabla K (s,r,\rho
w)\Big|^2S_{1}(dw) \big]^{1/2} d\rho \right)^2.       \label{eqn 2.6.1}
\end{align}
By \eqref{as2'} and the change of variable $(s-r)^{-\kappa}\rho \to \rho$, the last term is less than or equal to constant
times of
$$
 (s-r)^{-2{\sigma}+2\kappa(d+1)} \bM^{3r_1}_x |f|_H^2(r,x)  \left(\int_{ r_1 (s-r)^{-\kappa}  }^\infty  \rho^d
\big[\int_{\partial B_{1}} \Big| F_1 (s,r,\rho w)\Big|^2 S_{1}(dw) \big]^{1/2} d\rho \right)^2.
 $$
 By H\"older inequality and the definition of $H_{1,K}(\mu)$,
 \begin{align*}
&\left(\int_{ r_1 (s-r)^{-\kappa}  }^\infty  \rho^d\big[\int_{\partial B_{1}} \Big| F_1 (s,r,\rho w)\Big|^2 S_{1}(dw) \big]^{1/2} d\rho \right)^2\\
&\leq \left(\int^{\infty}_{r_{1}(s-r)^{-\kappa}} \rho^{d+1-\mu}
~d\rho\right)\cdot
 \left(\int^{\infty}_{r_{1}(s-r)^{-\kappa}}\int_{\partial B_{1}}
  \rho^{\mu+d-1}\Big| F_{1} (s,r,\rho w)\Big|^2 S_{1}(dw)  d\rho\right)  \\
&\leq  N r_{1}^{d+2-\mu}(s-r)^{\kappa(\mu-d-2)} H_{1,K}(\mu) .
\end{align*}
Coming back to (\ref{eqn 2.6.1}) and remembering the definition of $\cG f$, we get
\begin{align*}
&   \int_{Q_{r_2,r_1}} |\cG f(s,y)|^2 ~dsdy\\
&\leq  N H_{1,K}(\mu) r_1^{2d+2-\mu} \int_{ - 2r_2}^0 \int_{ -10r_2}^{s} (s-r)^{-2{\sigma}+{\kappa}(\mu+d)} \bM^{3r_1}_x |f|_H^2(r,x)dr ds \\
&\leq  N H_{1,K}(\mu) r_1^{2d+2-\mu}  \int_{ - 10r_2}^0 \left[\int_{r }^0 (s-r)^{-2{\sigma}+{\kappa}(\mu+d)} ds\right] \bM^{3r_1}_x |f|_H^2(r,x) dr \\
&\leq  N H_{1,K}(\mu) r_1^{2d+2-\mu}  \int_{ - 10r_2}^0  (-r)^{-2{\sigma}+{\kappa}(\mu+d)+1} \bM^{3r_1}_x |f|_H^2(r,x) dr \\
&\leq  N H_{1,K}(\mu) r_1^{2d+2-\mu}   r_2^{-2{\sigma}+{\kappa}(\mu+d)+1} \int_{ - 10r_2}^0  \bM^{3r_1}_x |f|_H^2(r,x) dr.
\end{align*}
The lemma is proved.
\end{proof}

Recall that $\Theta(\theta,\vartheta):=\theta d-2\vartheta$.

\begin{lemma}
\label{lemma3-2}
Suppose that
\begin{align}            \label{as3-1'}
\big|D^2_xK(t,s,x) \big| \leq C \Big((t-s)^{-\sigma} \big|F_2(t,s,(t-s)^{{-\kappa}} x) \big| \wedge (t-s)^{{-c}} \Big)
\end{align}
holds with some constants $\sigma,\kappa,c>0$ and there exists $\delta > 0$ such that
\begin{align*}
r_2^{\delta} = r_1,\quad \Theta(2\delta, c-\delta)<-1.
\end{align*}
Moreover  assume that there exists $\mu>d+2$ so that
 \begin{equation}               \label{theta2range2}
\Theta(\kappa+\delta,\sigma-\delta)-(\delta-\kappa)\mu < -1,
\end{equation}
and
$$
H_{2,K}(\mu):=\sup_{r\leq s}\int_{|x| \geq (s-r)^{\delta-\kappa} } |x|^{\mu}| F_2(s,r,x)|^2 dx<\infty.
$$
Let $f\in C^{\infty}_0(\fR^{d+1},H)$, and $f(t,x)=0$ for $t \geq -8r_2$. Then for any $(t,x) \in Q_{r_2,r_1}$ we have
\begin{align*}
\notag &\sup_{Q_{r_2,r_1}}  | \nabla \cG f|^2 \\
&\leq  N \Big(H_{2,K}(\mu)r_2^{\Theta(\kappa+\delta,\sigma-\delta)-(\delta-\kappa)\mu+1} \wedge r_{2}^{\Theta(2\delta,c-\delta)+1}\Big)\bM^{6r_2}_t \bM^{2r_1}_x |f|_H^2(t,x),
\end{align*}
where $N=N(d,\mu,\delta,c,\sigma,\kappa,C_0,C)$.
\end{lemma}

\begin{proof}
Let $(t,x),(s,y) \in Q_{r_2,r_1}$ and $r\leq s$.
By Minkowski's inequality
$$
\left|\frac{\|f(s+h, \cdot )\| - \|f(s, \cdot)\|}{h} \right|
\leq
\frac{\|f(s+h, \cdot ) - f(s, \cdot)\|}{|h|},
$$
the derivative of a norm is less than or
equal to the norm of the derivative  if both exist. Thus,
\begin{align*}
\Big|\frac{\partial}{\partial x_i}\cG f(s,y) \Big|
&=\Big|\frac{\partial}{\partial x_i} \Big( \int_{-\infty}^s  |K(s,r,\cdot) \ast f(r,\cdot)(y)|_H^2 dr \Big)^{1/2} \Big| \\
&\leq \Big(\int_{-\infty}^s \Big|\frac{\partial}{\partial x_i} K(s,r, \cdot) \ast  f(r,\cdot)(y)\Big|_H^2 dr\Big)^{1/2}.
\end{align*}
Applying Lemma \ref{lemma2}  with $R_{1}=2r_{1}$ and  $R_{2}=0$ we get
\begin{align*}
\Big|\frac{\partial}{\partial x_i} K(s,r,\cdot) \ast f(r, \cdot)(y)\Big|_H^2 \leq N\bM_{x}^{2r_1}|f|_H^2(r,x)(\cI_{1}^{2}+\cI_{2}^{2})
\end{align*}
where
$$
\cI_{1} = \int_{(s-r)^{\delta}}^\infty (2r_1 + \rho)^d \big(\int_{\partial B_{1}} \Big| D^2_x K (s,r,\rho w)\Big|^2S_{1}(dw) \big)^{1/2} d\rho,
$$
$$
\cI_{2} = \int_0^{(s-r)^{\delta}} (2r_1 + \rho)^d \big(\int_{\partial B_{1}} \Big| D^2_xK (s,r,\rho w)\Big|^2S_{1}(dw) \big)^{1/2} d\rho
$$

$\\$ Thus,
$$
\Big|\frac{\partial}{\partial x_i}\cG(s,y) \Big|^2 \leq N\int_{-\infty}^s \bM_{x}^{2r_{1}}|f|_{H}^{2}(r,x)(\cI_1^{2} + \cI_2^{2})~dr.
$$
Since $f(r,x)=0$ if $r \geq -8r_2$, we may assume $r < -8r_2$.
So
\begin{align}           \label{r2r1r}
|s-r|^\delta \geq 6r^\delta_2 =6r_1.
\end{align}
First, we estimate $\cI_1$. Due to \eqref{r2r1r} and \eqref{as3-1'},
\begin{align*}
\cI_1 &= \int_{(s-r)^{\delta}}^\infty (2r_1+ \rho)^d \big(\int_{\partial B_{1}} \Big| D^2_xK (s,r, \rho w) \Big|^2 S_{1}(dw) \big)^{1/2} d\rho \\
&\leq  N\int_{(s-r)^{\delta}}^\infty  \rho^d \big(\int_{\partial B_{1}} \Big| D^2_x K (s,r,\rho w)\Big|^2 S_{1}(dw) \big)^{1/2} d\rho \\
&\leq  N(s-r)^{-\sigma}\int_{(s-r)^{\delta}}^\infty  \rho^d \big(\int_{\partial B_{1}} |F_2 (s,r,(s-r)^{{-\kappa}} \rho w)|^2 S_{1}(dw) \big)^{1/2} d\rho.
\end{align*}
By  the change of variable $(s-r)^{-\kappa}\rho \to \rho$, the last therm  is less than or  equal to
\begin{align*}
&N(s-r)^{{-\sigma}+{\kappa} (d+1)}\int_{ (s-r)^{\delta-\kappa} }^\infty   \rho^d \big(\int_{\partial B_{1}} | F_2 (s,r,\rho w)|^2 S_{1}(dw) \big)^{1/2} d\rho \\
&\leq  N(s-r)^{{-\sigma}+{\kappa} (d+1)}
\left[\int_{(s-r)^{\delta-\kappa}}\rho^{d -\mu+1} d\rho \right]^{1/2}
 \left[\int_{|z|\geq (s-r)^{\delta-\kappa}}|z|^{\mu}|  F_2 (s,r,z)|^2 dz \right]^{1/2}\\
&\leq  N(s-r)^{\frac{\Theta(\kappa+\delta,\sigma-\delta)-(\delta-\kappa)\mu}{2}} \left[\int_{|z|\geq (s-r)^{\delta-\kappa}}|z|^{\mu}|  F_2 (s,r,z)|^2 dz \right]^{1/2}\\
&\leq  N H_{2,K}^{1/2}(\mu) (s-r)^{\frac{\Theta(\kappa+\delta,\sigma-\delta)-(\delta-\kappa)\mu}{2}}.
\end{align*}

Note $|s-r|\geq r/2$ for $r\leq -8r_2$. Thus, by the integration by parts and the assumption on $f$,
\begin{align*}
&\int_{-\infty}^s \bM_{x}^{2r_{1}}|f|_{H}^{2}(r,x)\cI_1^{2}~dr\\
&\leq  NH_{2,K}(\mu)\int_{-\infty}^{-8 r_2} (s-r)^{\Theta(\kappa+\delta,\sigma-\delta)-(\delta-\kappa)\mu}\bM_{x}^{2r_1} |f|_H^2(r,x)~dr\\
&\leq  NH_{2,K}(\mu)\int_{-\infty}^{-8 r_2} |r|^{\Theta(\kappa+\delta,\sigma-\delta)-(\delta-\kappa)\mu-1} \left[\int_r^0\bM_{x}^{2r_1} |f|_H^2(\bar{s},x)~d\bar{s}\right]dr\\
&\leq  N H_{2,K}(\mu)\bM^{6r_2}_t \bM_{x}^{2r_1} |f|_H^2(t,x) \int_{-\infty}^{-8 r_2} |r|^{\Theta(\kappa+\delta,\sigma-\delta)-(\delta-\kappa)\mu} ~dr\\
&\leq  N H_{2,K}(\mu)r_2^{\Theta(\kappa+\delta,\sigma-\delta)-(\delta-\kappa)\mu+1} \bM^{6r_2}_t \bM_{x}^{2r_1} |f|_H^2(t,x).
\end{align*}

$\\$ Next we estimate $\cI_2$. Using \eqref{as3-1'} and \eqref{r2r1r},
\begin{align*}
\cI_2 &\leq \int_0^{(s-r)^{\delta}} \big(2 r_1+ \rho\big)^d \big[\int_{\partial B_{1}} \Big| D^2_xK (s,r, \rho w)\Big|^2 S_{1}(dw) \big]^{1/2} \,d\rho \\
&\leq  N(s-r)^{{-c}} \int_0^{(s-r)^{\delta}}  \big(2r_1 +\rho\big)^d d\rho \leq  N(s-r)^{\frac{\Theta(2\delta,c-\delta)}{2}}.
\end{align*}
Applying the integration by parts again, we obtain
\begin{align*}
\int_{-\infty}^s &\bM_{x}^{2r_{1}}|f|_{H}^{2}(r,x)\cI_2  ^{2}~dr \\
&\leq  N\int_{-\infty}^{-8 r_2} (s-r)^{\Theta(2\delta,c-\delta)}\bM_{x}^{2r_1} |f|_H^2(r,x)~dr\\
&\leq  N\int_{-\infty}^{-8 r_2} |r|^{\Theta(2\delta,c-\delta)-1} \left[\int_r^0\bM_{x}^{2r_1} |f|_H^2(\bar{s},x)~d\bar{s}\right]dr\\
&\leq  N \bM^{6r_2}_t \bM_{x}^{2r_1} |f|_H^2(t,x) \int_{-\infty}^{-8 r_2} |r|^{\Theta(2\delta,c-\delta)}   ~dr\\
&\leq  N r_2^{\Theta(2\delta,c-\delta)+1} \bM^{6r_2}_t \bM_{x}^{2r_1} |f|_H^2(t,x).
\end{align*}
Finally, we get
\begin{align*}
&\Big|\frac{\partial}{\partial x_i}\cG f(s,y)\Big|^2\\
&\leq N \Big(H_{2,K}(\mu)r_2^{\Theta(\kappa+\delta,\sigma-\delta)-(\delta-\kappa)\mu+1} \wedge r_2^{\Theta(2\delta,c-\delta)+1}\Big) \bM^{6r_2}_t \bM_{x}^{2r_1} |f|_H^2(t,x).
\end{align*}
The lemma is proved.
\end{proof}

\begin{lemma}
\label{lemma4}
Suppose that
$$
\left|\frac{\partial^2}{ \partial x \partial t}K(t,s,x) \right| \leq C \Big((t-s)^{{-\sigma}} \big| F_3(t,s,(t-s)^{{-\kappa}} x) \big| \wedge (t-s)^{{-c}}\Big)
$$
holds with some constants $\sigma,\kappa,c>0$ and there exists a constant $\delta >0 $ such that
\begin{align*}
r_2^{\delta} = r_1,\quad \Theta(2\delta,c-\delta)<-1.
\end{align*}
Moreover  assume that there exists $\mu>d+2$ so that
\begin{align*}
\Theta(\kappa+\delta,\sigma-\delta)-(\delta-\kappa)\mu < -1,
\end{align*}
and
\begin{align*}
H_{3,K}(\mu):=\sup_{r\leq s}\int_{|x| \geq (s-r)^{\delta-\kappa} } |x|^{\mu}| F_3(s,r,x)|^2 dx<\infty.
\end{align*}
Let $f\in C^{\infty}_0(\fR^{d+1},H)$ and $f(t,x)=0$ for $t \geq -8r_2$. Then for any $(t,x) \in Q_{r_2,r_1}$ we have
\begin{align*}
\notag & \sup_{Q_{r_2,r_1}}  |D_t \cG f|^2   \\
 &\leq  N \Big( H_{3,K}(\mu)r_2^{\Theta(\kappa+\delta,\sigma-\delta)-(\delta-\kappa)\mu+1} \wedge  r_2^{\Theta(2\delta,c-\delta)+1}\Big) \bM^{6r_2}_t \bM_{x}^{2r_1} |f|_H^2(t,x),
\end{align*}
where $N=N(d,\mu,\delta,c,\sigma,\kappa,C_{0},C)$.
\end{lemma}

\begin{proof}
The proof of this lemma is quite similar to the previous one. Note that by Minkowski's inequality
\begin{align*}
|D_s\cG f(s,y)|
&= \big|D_s[\int_{-\infty}^{-8r_2}  |K(s,r,\cdot) \ast f(r,\cdot)(y)|_H^2 dr]^{1/2}\big| \\
&\leq [\int_{-\infty}^{-8r_2} |D_s K(s,r, \cdot) \ast  f(r,\cdot)(y)|_H^2 dr]^{1/2}.
\end{align*}
The other parts are easily obtained by following the proof of the previous lemma.
\end{proof}

\mysection{Proof of Theorem \ref{maintheorem}}              \label{pfmainthm}

First, observe that  from \eqref{c1c2relation} and \eqref{algebraic condition} we have
\begin{align}                \label{theta1range1}
-2{\sigma_1} + \kappa_1 (\mu_{1}+d) > -1.
\end{align}
Indeed,
\begin{align*}
-2{\sigma_1} + \kappa_1 (\mu_{1}+d)&=\mu_{1}(c_{2}-c_{3}+1)-d(c_{2}-c_{3}+1)-2(c_{2}-c_{3})-3
\\&=(c_{2}-c_{3}+1)(\mu_{1}-d-2)-1>-1,
\end{align*}
since $c_{2}-c_{3}+1 = \frac{2c_{2}-1}{2(d+2)}>0$ and $\mu_{1}>d+2$.

Also, we can derive the following relation from \eqref{c1c2relation} (note that $c_{2}>\frac{1}{2}$)
\begin{align}\label{delta0 equality}
\Theta(2\delta_{0},c_{2}-\delta_{0})=-2\delta_{0}-1=\frac{1-2c_{2}}{d+2}-1<-1
\end{align}
and
\begin{align}                    \label{theta3range1}
\Theta(2\delta_{0},c_{3}-\delta_{0})=\Theta(2\delta_{0},c_{2}-\delta_{0})+2(c_{2}-c_{3})=-3.
\end{align}

\vspace{3mm}

Take $\delta_{0}$  from \eqref{delta01}.  If Assumption \ref{assumption2} holds, then $\delta_0 >0$ due to \eqref{delta0 equality}. For $R>0$ set
$$
Q_R = (-2R,0)\times B_{R^{\delta_{0}}}.
$$

By $\aint_{Q_R}f~dsdy$ we denote the mean average of $f$ on $Q_R$, i.e.
$$
\aint_{Q_R} f~dsdy := \frac{1}{|Q_R|} \int_{Q_R} f(s,y)~dsdy.
$$
Recall
\begin{align*}
\cG f (t,x):=  \Big(\int_{-\infty}^t \big|K(t,s, \cdot) \ast f(s,\cdot)(x)\big|_H^2 ds\Big)^{1/2}.
\end{align*}

To continue the proof we need the following lemma.

\begin{lemma}               \label{lemma5}
Suppose that Assumption \ref{assumption1} and \ref{assumption2} hold. Then  for any $(t,x) \in Q_R$
\begin{align*}
\frac{1}{|Q_R|^2}\int_{Q_R} \int_{Q_R}\big| \cG f(s,y)-\cG f(r,z)\big|^2~dsdydrdz \leq N\bM_t\bM_{x} |f|_H^2(t,x),
\end{align*}
where the constant $N$ is independent of $f$, $R$, and $(t,x)$.
\end{lemma}

\begin{proof}
Let $(t,x) \in Q_R$. We take a function $\zeta \in C_0^\infty(\fR)$
such that $0\leq \zeta\leq 1$, $\zeta=1$ on $[-8R,8R]$, and $\zeta=0$ outside of $[-10R,10R]$. Define
$$
\cA(s,y) := f(s,y) \zeta(s), \quad
\cB(s,y):=f(s,y)-\cA(s,y)=f(s,y)(1-\zeta(s)).
$$
Then
$$
K(t,s,\cdot) \ast \cA(s,\cdot) = \zeta(s) K(t,s,\cdot) \ast f(s,\cdot),  \quad \cG
f \leq \cG \cA + \cG \cB\quad \text{and} \quad \cG \cB \leq \cG f.
$$ The first inequality comes from Minkowski's inequality. The second inequality comes from the fact $|K(t,s, \cdot) \ast \cB(s,\cdot)(y)|=(1-\zeta(s))|K(t,s, \cdot) \ast f(s,\cdot)(y)|$ and $|1-\zeta(s)| \leq 1$.
So for any constant $c$,
\begin{align}            \label{averaging}
|\cG f -c | \leq |\cG\cA| + |\cG \cB -c|.
\end{align}
This is because if $\cG f  \geq c$, then
$$
|\cG f -c | =\cG f -c  \leq \cG \cA + \cG \cB -c \leq |\cG\cA| + |\cG \cB -c|
$$
and if $\cG f  < c$, then
$$
|\cG f -c | =c- \cG f  \leq c - \cG \cB  \leq |\cG\cA| + |\cG \cB -c|.
$$
First we prove
\begin{align}            \label{localestimate1}
\int_{Q_R}|\cG \cA (s,y)|^2~dsdy \leq N |Q_R| \bM_t\bM_{x}|f|^2_H(t,x).
\end{align}
Take  $\eta \in C_0^\infty (\fR^d)$ such that $0 \leq \eta \leq 1$, $\eta =1$ in
$B_{2R^{\delta_0}}$, and $\eta=0$ outside of $B_{3R^{\delta_0}}$.  Set $\cA_{1} = \eta \cA$ and
$\cA_2 = (1-\eta) \cA$. By Minkowski's inequality, $\cG \cA \leq \cG \cA_1+
\cG \cA_2$. $\cG \cA_1$ can be estimated by Corollary \ref{cor1}.
Indeed,
\begin{align*}
\int_{-2R}^0 \int_{B_{R^{\delta_0}}}|\cG \cA_1(s,y)|^2dsdy
&\leq N  \int_{-\infty}^0 \int_{B_{3R^{\delta_0}}}|\cA_1(s,y)|_H^2dsdy\\
&\leq N  \int_{-10R}^0 \int_{B_{4R^{\delta_0}}(x)}|\cA_1(s,y)|_H^2dsdy\\
&\leq N R^{\delta_0 d} \int_{-10R}^0 \bM_{x}|\cA_1(s,x)|_H^2dsdy\\
&\leq N R^{1+\delta_0 d} \bM_t\bM_{x}|\cA_1(t,x)|_H^2\\
&\leq N R^{1+\delta_0 d} \bM_t\bM_{x}|f(t,x)|_H^2.
\end{align*}
Hence it only remains to show \eqref{localestimate1} for $\cG \cA_2$ instead of $\cG \cA$.

Due to \eqref{theta1range1}, \eqref{theta1range2} holds for $\mu=\mu_1$ and $(\sigma,\kappa)=(\sigma_{1},\kappa_{1})$. Thus from Lemma \ref{lemma3} with $(r_2,r_1)= (R,R^{\delta_0})$ we have
\begin{align*}
\int_{Q_R}|\cG \cA_2(s,y)|^2~dsdy &\leq N R^{\delta_0(2d+2-\mu_{1})-2{\sigma_1} + {\kappa_1}(\mu_{1}+d)+2} \bM_t\bM_{x}|f|^2_H(t,x)
\\&\leq N R^{\Theta(2\delta_{0}+\kappa_{1},\sigma_{1}-\delta_{0})+2-(\delta_{0}-\kappa_{1})\mu_{1}} \bM_t\bM_{x}|f|^2_H(t,x).
\end{align*}
Moreover due to \eqref{matrix} and \eqref{delta01},
$$
\Theta(2\delta_{0}+\kappa_{1},\sigma_{1}-\delta_{0})+2-(\delta_{0}-\kappa_{1})\mu_{1}=\Theta(\delta_{0},0)+1= \delta_0d+1
$$
and so \eqref{localestimate1} is obtained.
To go further, recall \eqref{matrix} and \eqref{delta0 equality},
$$\Theta(\kappa_{2}+\delta_{0},\sigma_{2}-\delta_{0})-(\delta_{0}-\kappa_{2})\mu_{2}=\Theta(2\delta_{0},c_{2}-\delta_{0})=-2\delta_{0}-1<-1$$
so \eqref{theta2range2} holds with $\mu = \mu_2$ and $(\sigma,\kappa,c)=(\sigma_{2},\kappa_{2},c_{2})$. Hence applying Lemma \ref{lemma3-2} with $(r_2,r_1)=(R,R^{\delta_0})$,
\begin{align*}
\notag &\sup_{Q_{R}}  |\nabla \cG \cB|^2 \\
&\leq N \Big( R^{\Theta(\kappa_{2}+\delta_{0},\sigma_{2}-\delta_{0})-(\delta_{0}-\kappa_{2})\mu_{2}+1} \wedge R^{\Theta(2\delta_{0},c_{2}-\delta_{0})+1}\Big)\bM_t \bM_{x} |\cB|_H^2(t,x)
\\&\leq NR^{-2\delta_{0}}\bM_t \bM_{x} |\cB|_H^2(t,x).
\end{align*}
Hence
\begin{align}           \label{localestimate2}
\sup_{Q_{R}}  |R^{\delta_0}  \nabla \cG \cB|^2 \leq N\bM_t \bM_{x} |\cB|_H^2(t,x).
\end{align}
Similarly Lemma \ref{lemma4} with $(r_2,r_1)=(R,R^{\delta_0})$, $(\mu,\delta,\sigma,\kappa,c)=(\mu_{3},\delta_{0},\sigma_{3},\kappa_{3},c_{3})$ gives
\begin{align}               \label{localestimate3}
\sup_{Q_{R}}  |R  \frac{\partial}{\partial t} (\cG \cB)|^2 \leq N\bM_t \bM_{x} |\cB|_H^2(t,x).
\end{align}
To apply Lemma \ref{lemma3-2} and Lemma \ref{lemma4} above we used the fact that $\cG \cB(s,y) = \cG (I_{(-\infty,0)}\cB)(s,y)$ on $Q_R$. Next by  \eqref{averaging},
\begin{align*}
\frac{1}{|Q_R|^2}\int_{Q_R} &\int_{Q_R}\big| \cG f(s,y)-\cG f(r,z)\big|^2~dsdydrdz \\
&\leq 2 \dashint_{Q}|\cG f-c|^2~dsdy \leq 4 \dashint_{Q} |\cG \cA|^2~dsdy + 4\dashint_{Q} |\cG \cB - c |^2~dsdy.
\end{align*}
Taking $c=\cG \cB(t,x)$, from \eqref{localestimate1}, \eqref{localestimate2}, and \eqref{localestimate3}  we get
\begin{align*}
\frac{1}{|Q_R|^2}\int_{Q_R} &\int_{Q_R}\big| \cG f(s,y)-\cG f(r,z)\big|^2~dsdydrdz \\
&\leq 4\dashint_{Q_R} |\cG \cA|^2~dsdy + 4\dashint_{Q_R} |\cG \cB - \cG \cB (t,x) |^2~dsdy\\
&\leq N \bM_t \bM_{x} |f|_H^2 (t,x)+4\dashint_{Q_R} |\cG \cB - \cG \cB (t,x) |^2~dsdy \\
&\leq N \bM_t \bM_{x} |f|_H^2 (t,x)+N\sup_{Q_R} \big(|RD_s\cG \cB|^2 + |R^{\delta_0}\nabla_\cG \cB|^2\big)\\
&\leq N \bM_t \bM_{x} |f|_H^2 (t,x).
\end{align*}
The lemma is proved.
\end{proof}

We continue the proof of the theorem.
For  measurable functions $h(t,x)$ on $\fR^{d+1}$, we define the sharp function $h^\sharp(t,x)$
\begin{align*}
h^\sharp(t,x)
= \sup_{Q }  \frac{1}{|Q|}  \int_{Q} |f(r,z) - f_{Q}|~drdz,
\end{align*}
where $f_{Q} := \frac{1}{|Q|}  \int_{Q} f(r,z)~drdz$, and the sup is taken all $Q$ containing $(t,x)$ of the type
$$
Q= (s-R,s+R) \times B_{R^{\delta_0}}(y), \quad  R>0.
$$

By Fefferman-Stein Theorem \cite[Theorem 4.2.2]{Ste}, for any $h\in L_p(\fR^{d+1})$,
$$
\|h\|_{L_p(\fR^{d+1})} \leq N \|h^{\sharp}\|_{L_p(\fR^{d+1})}.
$$

\vspace{3mm}

Now we claim
\begin{equation}
                        \label{eqn 6.08.9}
  (\cG f)^{\sharp}(t,x)\leq N (\bM_t\bM_{x}|f|^2_{H})^{1/2}(t,x).
  \end{equation}
By Jensen's inequality, to prove (\ref{eqn 6.08.9}) it suffices to
prove that for each $Q \in \cQ$ and $(t,x)\in Q$,
$$
\dashint_Q |\cG f- (\cG f)_Q|^2~dyds \leq
N \bM_t\bM_{x}|f|_H^2(t,x).
$$

Note that for any $h_1 \in \fR$ and $h_2 \in \fR^d$,
\begin{align*}
\cG f(t-h_1,x-h_2)
=\cG \tilde f (t,x)
=\Big(\int_{-\infty}^t \big|\tilde K(t,s, \cdot) \ast \tilde f(s,\cdot)(x)\big|_H^2 ds \Big)^{1/2},
\end{align*}
where $\tilde f (t,x) = f(t-h_1, x-h_2)$ and $\tilde K (t,s,y)= K(t-h_{1},s-h_{1},y)$.
Since $\tilde K$ also satisfies Assumptions \ref{assumption1} and \ref{assumption2} with the same constnats, we may assume
$Q=[-2R,0] \times B_{R^{\delta_0 }}$.
Thus Lemma \ref{lemma5} proves (\ref{eqn 6.08.9}) because
$$
\dashint_Q |\cG f- (\cG f)_Q|^2~dyds \leq \frac{1}{|Q_R|^2}\int_{Q_R} \int_{Q_R}\big| \cG f (s,y)- \cG f (r,z)\big|^2~dsdydrdz.
$$

Finally, combining the Fefferman-Stein theorem and Hardy-Littlewood
maximal theorem \cite[Theorem 1.3.1]{Ste}, we conclude (recall
$p/2
>1$)
\begin{align*}
\|u\|_{L_p(\fR^{d+1})}^p \leq N \|(\bM_t\bM_{x}|f|_H^2)^{1/2}\|_{L_p(\fR^{d+1})}^p
&=N \int_{\fR^d} \int_\fR (\bM_t \bM_{x} |f|_H^2)^{p/2}dtdx \\
&\leq N \int_{\fR^d}  \int_\fR (\bM_{x} |f|_H^2 )^{p/2}dtdx \\
&= N \int_{\fR}  \int_{\fR^d} (\bM_{x} |f|_H^2 )^{p/2}dxdt \\
&\leq  N \|f\|_{L_p(\fR^{d+1},H)}^p.
\end{align*}
Therefore, the theorem is proved. \qed

\mysection{ Proof of Theorem \ref{pseudothm} }
                                        \label{pseudothmpf}
Denote
$$
K(t,s,x)=(-\Delta)^{\gamma/4} p(t,s,x) = I_{0 \leq s<t }\cF^{-1} \Big(|\xi|^{\gamma/2} \exp\big\{\int^t_s \psi(r,\xi)dr\big\} \Big)(x).
 $$
We prove that Assumptions \ref{assumption1} and \ref{assumption2} hold with

$$
F_{1}(t,s,x) =I_{0 \leq s<t} \sum_{i}\left|\cF^{-1} \Big( \xi^i |\xi|^{\gamma/2} \exp\big\{M(t,s,\xi)\big\} \Big)(x)\right|,
$$
$$
F_{2}(t,s,x) = I_{0 \leq s<t }\sum_{i,j}\left|\cF^{-1} \Big( \xi^i \xi^j |\xi|^{\gamma/2} \exp\big\{M(t,s,\xi)\big\} \Big)(x)\right|,
$$
and
\begin{align*}
F_{3}(t,s,x)
= I_{0 \leq s<t }\sum_i\left|\cF^{-1} \Big( (t-s)\psi\Big(t,\frac{\xi}{(t-s)^{1/\gamma}}\Big) \xi^i |\xi|^{\gamma/2} \exp\big\{M(t,s,\xi)\big\} \Big)(x)\right|,
\end{align*}
where $
{M}(t,s,\xi):=\int_s^t \psi\Big(r,\frac{\xi}{(t-s)^{1/\gamma}}\Big) dr.
$

 In the the  following lemma  we  first prove (\ref{as1})-(\ref{as3-2}) with
 $$
 \kappa_1=\kappa_2=\kappa_3=\gamma^{-1}, \quad \sigma_1=\frac{d+1}{\gamma}+\frac{1}{2}, \quad \sigma_2=c_2=\frac{d+2}{\gamma}+\frac{1}{2}, \quad \sigma_3=c_3=\frac{d+1}{\gamma}+\frac{3}{2}.
 $$
\begin{lemma}           \label{confirmas1}
There exists a constant $N=N(d,\gamma,\nu)>0$ such that
\begin{align*}
\int_0^\infty \Big| \cF \Big(K(t,s,\cdot) \Big)(\xi)\Big|^2 dt < N,
\end{align*}
\begin{align*}
\big|D_x K(t,s,x) \big| \leq N (t-s)^{-\frac{d}{\gamma} - \frac{1}{2} -\frac{1}{\gamma}} \Big(|F_{1}(t,s,x)| \wedge 1 \Big),
\end{align*}
\begin{align*}
\big|D^2_x
K(t,s,x) \big|
\leq N(t-s)^{-\frac{d}{\gamma} - \frac{1}{2} -\frac{2}{\gamma}} \Big(|F_{2}(t,s,x)| \wedge 1 \Big) ,
\end{align*}
and
\begin{align*}
\big|\frac{\partial}{ \partial t }D_xK(t,s,x) \big|
\leq N  (t-s)^{-\frac{d}{\gamma} -\frac{3}{2} -\frac{1}{\gamma}}\Big(|F_{3}(t,s,x)| \wedge 1 \Big).
\end{align*}
\end{lemma}
\begin{proof}
The first assertion comes from (\ref{A1}). Indeed, since $\Re \psi(t,\xi) \leq - \nu |\xi|^{\gamma}$,
\begin{align*}
&\int_0^\infty \Big| \cF \Big(K(t,s,\cdot) \Big)(\xi)\Big|^2 dt=\int^{\infty}_s \left||\xi|^{\gamma/2}  \exp\big\{\int^t_s \psi(r,\xi)dr\big\}\right|^2 dt\\
&\leq N\int_0^\infty |\xi|^{\gamma} e^{-2\nu t|\xi|^{\gamma}} dt \leq N.
\end{align*}
Next because of the similarity, we only prove the last assertion. From the definition of $K(t,s,x)$ and $M(t,s,\xi)$,
\begin{align*}
&\Big|\frac{\partial}{ \partial x^i }\frac{\partial}{ \partial t}K(t,s,x) \Big|
= I_{0 \leq s<t } \Big|\cF^{-1} \Big( \psi(t,\xi) \xi^i |\xi|^{\gamma/2}  \exp\big(\int_s^t \psi(r,\xi) dr \big) \Big)(x) \Big| \\
&= I_{0 \leq s<t } (t-s)^{-\frac{d}{\gamma} -\frac{3}{2} -\frac{1}{\gamma}} \\
& \quad \quad \quad \Big|\cF^{-1} \Big( (t-s)\psi\Big(t,\frac{\xi}{(t-s)^{1/\gamma}}\Big) \xi^i |\xi|^{\gamma/2} \exp\big\{M(t,s,\xi)\big\} \Big)\Big(\frac{x}{(t-s)^{1/\gamma}}\Big)\Big|\\
&\leq (t-s)^{-\frac{d}{\gamma} -\frac{3}{2} -\frac{1}{\gamma}}F_3(t,s,(t-s)^{-1/{\gamma}}x).
\end{align*}
Furthermore, by   (\ref{A1}) and (\ref{A2}),
\begin{align*}
&\Big|\cF^{-1} \Big( (t-s)\psi\Big(t,\frac{\xi}{(t-s)^{1/\gamma}}\Big) \xi^i |\xi|^{\gamma/2} \exp\big\{M(t,s,\xi)\big\} \Big)(x)\Big| \\
&\leq N \int_{\fR^d} \Big| (t-s)\psi\Big(t,\frac{\xi}{(t-s)^{1/\gamma}}\Big) \xi^i |\xi|^{\gamma/2}  \exp\big(\int_s^t \psi(r,\frac{\xi}{(t-s)^{1/\gamma}}) dr \big) \Big|~d\xi\\
&\leq N \int_{\fR^d}  |\xi|^{\frac{3\gamma}{2} +1}  \exp\big(-\nu |\xi|^\gamma \big) ~d\xi \leq N.
\end{align*}
Hence the assertion is proved.
\end{proof}

\begin{lemma}   \label{eta1}
Let   $h \in C^2(\fR^d \setminus \{0\})$ satisfy
\begin{align*}
|h(x)| \leq N_0|x|^\varsigma e^{-c|x|^\gamma},  \quad \forall x\in\fR^{d} \setminus \{0\},
\end{align*}
with some    constants $c, N_0>0$, $\varsigma > \eta-\frac{d}{2}$ and $\gamma > 0$. Further assume that either
\begin{align*}
 \eta\in [0,1)\quad \text{and}\quad \big|D h(x) \big| \leq N_0|x|^{\varsigma-1} e^{-c|x|^\gamma},  \quad \forall x \in \fR^{d} \setminus \{0\}
\end{align*}
 or
 \begin{equation*}
\eta\in [1,2) \quad \text{and}\quad \big|D^2h(x) \big| \leq N_0|x|^{\varsigma-2} e^{-c|x|^\gamma},  \quad \forall x \in \fR^{d} \setminus \{0\}
\end{equation*}
 holds.
 Then
$$
\|(-\Delta)^{\eta/2}h\|_{L_2(\fR^d)} < N< \infty,
$$
where  $N=N(N_0,\eta,c,\varsigma,\gamma)$.
\end{lemma}
\begin{proof}
See \cite[Lemma 5.1]{IKS}.
\end{proof}

\begin{corollary}           \label{confirmas2}
 Suppose
$$
\begin{cases}
&2\lfloor\frac{\mu}{4}\rfloor +1 \leq \lfloor \frac{d}{2} \rfloor+2 \quad \text{if} \quad \frac{\mu}{2} - 2\lfloor \frac{\mu}{4} \rfloor \in [0,1)\\
&2\lfloor\frac{\mu}{4}\rfloor +2 \leq \lfloor \frac{d}{2} \rfloor+2 \quad \text{if} \quad \frac{\mu}{2} - 2\lfloor \frac{\mu}{4} \rfloor \in [1,2).
\end{cases}
$$
Then,
\begin{align*}
&\sup_{s<t}\int_{\fR^d} |x|^{\mu} |F_1(t,s,x)|^2 dx  < \infty, \quad \text{if} \quad  \mu <  \gamma+d+2;\\
&\sup_{s<t}\int_{\fR^d} |x|^{\mu} |F_2(t,s,x)|^2 dx  < \infty, \quad \text{if}\quad  \mu <  (\gamma+d+4);\\
& \sup_{s<t}\int_{\fR^d} |x|^{\mu} |F_3(t,s,x)|^2 dx  < \infty,  \quad \text{if}\quad \mu < (3\gamma+d+2).
\end{align*}
\end{corollary}

\begin{proof}
Because of the similarity of proofs, we only prove the last assertion.
By Parseval's identity, it suffices to show
\begin{align*}
\sup_{s<t}\int_{\fR^d} \Big|(-\Delta)^{\mu/4} \tilde{F}^i_3(t,s,\xi)\Big|^2 d\xi  < \infty, \quad \forall\, i,
\end{align*}
where
\begin{align*}
\tilde{F}^i_3(t,s,\cdot)\big)(\xi)
= I_{0 \leq s<t}(t-s)\psi\Big(t,\frac{\xi}{(t-s)^{1/\gamma}}\Big) \xi^i |\xi|^{\gamma/2} \exp\big\{M(t,s,\xi)\big\} .
\end{align*}
 Using (\ref{A1}) and (\ref{A2}), one can check that there exists a constant $N=N(\nu,m)$ such for each $0<s<t$, $\xi \neq 0$, and $\mu < \lfloor \frac{d}{2} \rfloor +2$
$$
|(-\Delta)^{\lfloor\mu/4\rfloor} \tilde{F}^i_3(s,t,\xi)| \leq N |\xi|^{\frac{3\gamma}{2}+1 - 2\lfloor\mu/4\rfloor}e^{-\nu |\xi|^{\gamma}}.
$$
Moreover
$$
|\frac{\partial^2}{\partial \xi^i }(-\Delta)^{\lfloor\mu/4\rfloor}\tilde F_3^i(s,t,\xi)| \leq N |\xi|^{\frac{3\gamma}{2} - 2\lfloor\mu/4\rfloor}e^{-\nu |\xi|^{\gamma}},
$$
if $\frac{\mu}{4} - \lfloor \frac{\mu}{4} \rfloor \in [0,1)$,
and
$$
|\frac{\partial^2}{\partial \xi^i \partial \xi^j}(-\Delta)^{\lfloor\mu/4\rfloor}\tilde F_3^i(s,t,\xi)| \leq N |\xi|^{\frac{3\gamma}{2}-1 - 2\lfloor\mu/4\rfloor}e^{-\nu |\xi|^{\gamma}}.
$$
if $\frac{\mu}{4} - \lfloor \frac{\mu}{4} \rfloor \in [1,2)$.
 Finally we set
 $$\eta = \mu/2-2\lfloor\mu/4\rfloor, \quad \varsigma =\frac{3\gamma}{2}+1-2\lfloor\mu/4\rfloor.
 $$
 Then, for   $ \mu < 3\gamma+d+2$,  we have
$$
\eta-\frac{d}{2}< \varsigma.
$$
Therefore Lemma \ref{eta1} is applicable, and the assertion is proved.
\end{proof}

We continue the proof of the theorem. Recall that we defined
$$
{\kappa_1}={\kappa_2}={\kappa_3}= \frac{1}{\gamma}, \quad {\sigma_1}= \frac{d}{\gamma}+\frac{1}{2}+\frac{1}{\gamma},
$$
$$
{c_{2}}={\sigma_2}=\frac{d}{\gamma}+\frac{1}{2}+\frac{2}{\gamma}, \quad {c_{3}}={\sigma_3}=\frac{d}{\gamma}+\frac{3}{2}+\frac{1}{\gamma}.
$$
So obviously
\begin{align*}
\delta_{0}={c_{2}}-c_{3}+1= \frac{1}{\gamma},\quad c_{2}>\frac{1}{2},
\end{align*}

\begin{align*}
\Theta(\kappa_{1}+\delta_{0},\sigma_{1}-\delta_{0})+1 = \frac{2d}{\gamma}-2\left(\frac{d}{\gamma}+\frac{1}{2}\right)+1=0,
\end{align*}

\begin{align*}
\Theta(\kappa_{2}-\delta_{0},\sigma_{2}-c_{2})=\Theta(0,0)=0,
\end{align*}
and
\begin{align*}
\Theta(\kappa_{3}-\delta_{0},\sigma_{3}-c_{3})=\Theta(0,0)=0.
\end{align*}
Thus \eqref{algebraic condition} (or equivalently (\ref{matrix})) is satisfied for any $(\mu_1,\mu_2,\mu_3)\in \fR^3$.
Next we choose $(\mu_{1},\mu_{2},\mu_{3})$ such that
$$
d+2 <\mu_1 < \gamma+d+2,
$$
$$
d+2 <\mu_2< \gamma+d+4,
$$
and
$$
d+2 <\mu_3< 3\gamma+d+2
$$
so that for all $1 \leq i \leq 3$
$$
\begin{cases}
&2\lfloor\frac{\mu_i}{4}\rfloor +1 \leq \lfloor \frac{d}{2} \rfloor+2 \quad \text{if} \quad \frac{\mu_i}{2} - 2\lfloor \frac{\mu_i}{4} \rfloor \in [0,1)\\
&2\lfloor\frac{\mu_i}{4}\rfloor +2 \leq \lfloor \frac{d}{2} \rfloor+2 \quad \text{if} \quad \frac{\mu_i}{2} - 2\lfloor \frac{\mu_i}{4} \rfloor \in [1,2).
\end{cases}
$$
Then due to Corollary \ref{confirmas2}, we see that
\eqref{finite1}, \eqref{finite2}, and \eqref{finite3} hold for these $\mu_1$, $\mu_2$, and $\mu_3$ hence Assumption \ref{assumption2} holds. The theorem is proved.

\end{document}